\documentclass[11pt,reqno]{amsart}
\usepackage[margin=1in]{geometry}
\usepackage{amsmath,amssymb,amsfonts,amsthm,mathtools}
\usepackage{xcolor}
\usepackage[pagebackref=true,colorlinks=true,linkcolor=blue,citecolor=blue,urlcolor=blue]{hyperref}
\usepackage{cite}
\renewcommand*{\backref}[1]{}
\renewcommand*{\backrefalt}[4]{\ifcase#1\relax\or\space#2\else\space#2\fi}

\allowdisplaybreaks[2]
\numberwithin{equation}{section}

\newtheorem{theorem}{Theorem}[section]
\newtheorem{lemma}[theorem]{Lemma}
\newtheorem{proposition}[theorem]{Proposition}
\newtheorem{corollary}[theorem]{Corollary}
\newtheorem{remark}{Remark}[section]

\makeatletter
\def\subsection{\@startsection{subsection}{2}%
  \z@{.5\linespacing\@plus.7\linespacing}{.25\linespacing}%
  {\normalfont\bfseries}}
\makeatother

\newcommand{\R}{\mathbb R}
\newcommand{\T}{\mathbb T}
\newcommand{\Z}{\mathbb Z}
\newcommand{\dd}{\,\mathrm d}
\newcommand{\supp}{\operatorname{supp}}
\newcommand{\Thmref}[1]{\hyperref[#1]{\mbox{Theorem~\ref*{#1}}}}
\newcommand{\Lemref}[1]{\hyperref[#1]{\mbox{Lemma~\ref*{#1}}}}
\newcommand{\Propref}[1]{\hyperref[#1]{\mbox{Proposition~\ref*{#1}}}}
\newcommand{\Corref}[1]{\hyperref[#1]{\mbox{Corollary~\ref*{#1}}}}
\newcommand{\Remref}[1]{\hyperref[#1]{\mbox{Remark~\ref*{#1}}}}
\newcommand{\Secref}[1]{\hyperref[#1]{\mbox{Section~\ref*{#1}}}}
\newcommand{\Subsecref}[1]{\hyperref[#1]{\mbox{Subsection~\ref*{#1}}}}
\title[Sharp critical Strichartz estimates on the waveguide]{Sharp critical Strichartz estimates for a resonant hyperbolic Schr\"odinger flow on the 3D waveguide}

\author[X. Cen]{Xi Cen}
\address{Xi Cen, School of Science, China University of Mining and Technology-Beijing, Beijing 100083, People's Republic of China}
\email{xicenmath@gmail.com}

\author[Z. Zhang]{Zhezhi Zhang}
\address{Zhezhi Zhang, School of Mathematics and Information Science, Henan Polytechnic University, Jiaozuo 454003, People's Republic of China}
\email{zhangzhezhi@home.hpu.edu.cn}


\begin{document}
\begin{abstract}
We study critical Strichartz estimates for the resonant hyperbolic Schr\"odinger flow
on the waveguide $\R\times\T^2$. For initial data localized to frequencies at most $N$, we prove that the unit-time $L^2\to L^{10/3}$ operator norm has sharp order $N^{1/5}$. We further track its dependence on the observation length, obtaining sharp estimates in the short- and long-time regimes and characterizing the frequency-dependent intervals on which the critical estimate remains lossless. The precise dependence on $T$ in the intermediate regime $N^{-1}<T<1$ remains open. The upper bounds are obtained by lifting the global endpoint estimate on the elliptic cylinder $\R\times\T$ to Hilbert-valued data, preserving its time summability, and combining this with periodic Bernstein and interpolation. The lower bounds exploit an exact integer null direction in the periodic frequency lattice together with a suitably scaled Euclidean wave packet. Finally, using a Bloch--Floquet transference argument, we recover the known elliptic unit-time estimate from compact torus theory and obtain a lossless elliptic estimate on frequency-dependent short time intervals.
\end{abstract}
\maketitle

\section{Introduction}\label{sec:intro}
\enlargethispage{3pt}
\subsection{Background and motivations}\label{sub:background}

We use the normalization $\T=\R/(2\pi\Z)$.  Let 
$
 \mathcal L_h:=\partial_x^2+\partial_{y_1}^2-\partial_{y_2}^2
$
be the hyperbolic Schr\"odinger operator on $\R\times\T^2$.  For
$f\in L^2(\R\times\T^2)$ we consider
\begin{equation}
 \begin{cases}
  (i\partial_t+\mathcal L_h)u=0,\\
  u|_{t=0}=f.
 \end{cases}
 \label{eq:main-hyperbolic-flow}
\end{equation}
The product $\R\times\T^2$ is a basic semiperiodic waveguide: the Euclidean
direction is genuinely dispersive, while the periodic directions allow
recurrence, coherent lattice packets, and arithmetic concentration.  Such
geometries arise naturally in partially confined dispersive systems
\cite{TzvetkovVisciglia2012,HaniPausader2014,DWWZ2025}.  Mixed-signature
Schr\"odinger operators also occur in anisotropic models from nonlinear optics
and water-wave modulation. See \cite{SautWang2024}.  Analytically, the space
$\R\times\T^2$ lies at a particularly delicate interface.  Its single
Euclidean direction provides only one source of long-time decay, whereas the
two periodic directions already support genuinely two-dimensional lattice
interactions.  The problem is therefore governed by a competition among
Euclidean dispersion, periodic confinement, and the arithmetic degeneracy of
an indefinite quadratic phase.

The diagonal spacetime exponent is dictated by three-dimensional
Schr\"odinger scaling.  In dimension $d$, the $L^2$-critical diagonal exponent
is
\[
 p_c=\frac{2(d+2)}{d},
\]
and hence $p_c=10/3$ for $d=3$.  Equivalently, the Euclidean frequency
exponent $d/2-(d+2)/p$ vanishes at $p=10/3$.  In Euclidean space this critical
scale is compatible with the lossless Strichartz estimate originating in
\cite{Strichartz1977} and the general dispersive framework of Keel--Tao
\cite{KeelTao1998}.  On compact spaces, by contrast, global dispersion is
absent and the same exponent becomes sensitive to lattice arithmetic.
Bourgain's periodic restriction theory \cite{Bourgain1993} initiated this
viewpoint. Bourgain--Demeter \cite{BourgainDemeter2015} obtained the critical
estimate up to the characteristic $N^\varepsilon$ loss through decoupling,
while Killip--Vi\c{s}an \cite{KillipVisan2016} established scale-invariant
estimates away from the endpoint.  Thus any polynomial frequency growth at
$p=10/3$ reflects concentration not predicted by Euclidean scaling alone.

The same geometric sensitivity persists on waveguides.  Takaoka--Tzvetkov
\cite{TakaokaTzvetkov2001} proved the sharp local $L^4$ estimate on
$\R\times\T$, and subsequent works such as Herr--Tataru--Tzvetkov
\cite{HTT2014} and Hani--Pausader \cite{HaniPausader2014} developed
higher-dimensional semiperiodic estimates and their nonlinear applications.
Barron \cite{Barron2021} later established a general global-in-time theory on
$\R^m\times\T^n$.  At the critical exponent, that endpoint theory permits an
arbitrarily small Sobolev loss. For $(m,n)=(1,2)$, frequency localization
therefore yields the familiar $N^\varepsilon$ bound at $p=10/3$.  Recent
bilinear and restricted-type results
\cite{DFYZZ2024,Deng2025,DWWZ2025,DDFZ2026} further show that critical
behavior depends sensitively on how the total dimension is divided between
Euclidean and periodic variables.  In particular, the one-Euclidean,
two-periodic geometry studied here is not a formal variant of either
$\R^2\times\T$ or the fully compact problem.

A second source of difficulty is the length of the observation interval.
Besides asking for the optimal frequency loss on a fixed interval, one may
ask for the longest frequency-dependent interval on which a lossless estimate
remains uniform.  Recent short-time results, including
\cite{Schippa2025,McConnell2026,Quinn2026,SkouloudisYu2026}, show that these
two formulations often reveal complementary aspects of the same endpoint
phenomenon.  For the present paper, the decisive global input is the sharp
cylinder estimate of Barron--Christ--Pausader \cite{BCP2021},
\begin{equation}
 \left(
 \sum_{\gamma\in\Z}
 \|e^{it\Delta_{\R\times\T}}f\|_
 {L^4([\gamma,\gamma+1]\times\R\times\T)}^8
 \right)^{1/8}
 \lesssim \|f\|_{L^2(\R\times\T)}.
 \label{eq:bg-bcp}
\end{equation}
The $\ell^8$ summability in \eqref{eq:bg-bcp}, rather than only its unit-time
consequence, records how the local cylinder norms accumulate over long
intervals.  This summability survives the Hilbert-valued lifting used below
and is what makes it possible to obtain a sharp long-time law instead of a
bound produced by naive iteration.  Long-time estimates on
three-dimensional elliptic waveguides have also been developed recently in
\cite{DDMYY2026}, underscoring that the time length is an independent
parameter in semiperiodic Strichartz theory.

The hyperbolic problem introduces a further arithmetic obstruction.  Wang
\cite{Wang2013} exhibited the role of difference-of-squares arithmetic on
$\T^2$, while restriction estimates for indefinite paraboloids reflect the
same distinction at the Fourier-analytic level \cite{DemeterWu2025}.  On the
cylinder $\R\times\T$, Ba\c{s}ako\u{g}lu--Sun--Tzvetkov--Wang
\cite{BSTW2025} obtained sharp estimates for $p>4$, and
Deng--Fan--Zhao \cite{DFZ2025} proved the lossless critical $L^4$ estimate.
Consequently, indefinite signature by itself does not force a derivative
loss.  In three compact dimensions, however, Liu--Zheng
\cite{LiuZheng2025} obtained sharp hyperbolic Strichartz
estimates with a genuine critical loss. At $p=10/3$ the sharp frequency
growth is $N^{1/5}$.  For the standard signature considered here, their
compact unit-time upper bound can also be transferred to $\R\times\T^2$ by
the Bloch--Floquet principle in \Propref{prop:transfer}.  Compact
theory alone, however, does not determine how the waveguide norm depends on
the observation length.  The principal issue here is therefore the
long-time growth law and the transition between lossless and lossy time
scales. Related estimates for symmetric hyperbolic flows in even compact
dimensions are proved in \cite{LiuZheng2026}.

The mechanism behind the lower bound is already visible in the periodic
lattice.  The dispersion relation
\[
 \omega(\xi,k_1,k_2)=\xi^2+k_1^2-k_2^2
\]
contains the exact integer null line
\begin{equation}
 \omega(0,j,j)=0,\qquad j\in\Z.
 \label{eq:bg-null-line}
\end{equation}
Along this line the two periodic phases cancel exactly, allowing a
Dirichlet-type packet to remain coherent.  The resulting stationary periodic
concentration gives the $N^{1/5}$ unit-time lower bound. When it is paired
with a slowly dispersing Euclidean packet, it also yields the matching
long-time lower bound.  This exact resonance is specific to the standard
coefficient choice in \eqref{eq:main-hyperbolic-flow} and need not persist for
a general indefinite diagonal form.

The upper bound is driven by the complementary dispersive structure.  We
regard $(x,y_1)$ as an elliptic cylinder, retain the $y_2$ Fourier variable as
an $\ell^2$ parameter, and lift \eqref{eq:bg-bcp} to Hilbert-valued data.
Periodic Bernstein in the remaining variable, followed by interpolation with
spacetime $L^2$, then produces the critical $L^{10/3}$ estimate with the
correct accumulation in time.  The sharp result thus reflects a balance
between two mechanisms: global cylinder dispersion controls the upper bound,
while the integer null direction supplies the concentration needed for
the lower bound.

The elliptic problem provides a useful benchmark.  Its unit-time
$N^\varepsilon$ estimate on $\R\times\T^2$ is already contained in Barron's
general endpoint theory. The elliptic theorem below is therefore included to present
an alternative Bloch--Floquet transfer proof from compact estimates rather
than as a new endpoint result.  For shorter windows we use Quinn's lossless
critical estimate on rectangular tori \cite[Theorem~1.1]{Quinn2026}.  In
spatial dimension three, Quinn's notation has $n=4$ and hence $q_c=10/3$. The condition $\delta\ge\lambda^{-1/5+\varepsilon}$ gives a lossless
interval of length $(\lambda\delta)^{-1}$, in particular of order
$\lambda^{-4/5-\varepsilon}$.  Fixed positive coefficients can be treated by
a constant spatial and time rescaling.  This elliptic comparison separates
the effects of semiperiodic geometry from the additional polynomial loss
created by the hyperbolic null lattice.

For the statements and proofs below, we use the diagonal notation
\begin{equation}
 L_a\coloneqq a_1\partial_x^2+a_2\partial_{y_1}^2+a_3\partial_{y_2}^2,
 \label{eq:operator}
\end{equation}
where the coefficients $a_j$ are fixed nonzero real numbers, and write
\[
 Q_a(\xi,k)\coloneqq a_1\xi^2+a_2k_1^2+a_3k_2^2.
\]
For the hyperbolic model considered above, we set
\begin{equation}
 L\coloneqq L_{(1,1,-1)}=\mathcal L_h.
 \label{eq:hyperbolic-normal}
\end{equation}
Thus $Q\coloneqq Q_{(1,1,-1)}$ agrees with the dispersion relation $\omega$
introduced above.

\subsection{Main results}\label{sec:main-results}

We fix an even $\chi\in C_c^\infty(\R)$ taking values in $[0,1]$, equal to one on $[-1,1]$, and supported in $[-2,2]$. For dyadic $N\ge2$, we define the frequency cutoff by
\[
 \widehat{P_{\le N}f}(\xi,k)
 \coloneqq\chi\bigl(N^{-1}\sqrt{\xi^2+|k|^2}\bigr)\widehat f(\xi,k).
\]
The plateau of $\chi$ will allow us to test the estimate on explicit initial data without changing those data under the cutoff.

For the fixed hyperbolic operator $L$ in \eqref{eq:hyperbolic-normal}, set
\[
 \mathcal S(N,T)\coloneqq
 \left\|e^{itL}P_{\le N}\right\|_{L^2(\R\times\T^2)\to
 L^{10/3}([0,T]\times\R\times\T^2)}.
\]

\begin{theorem}[Hyperbolic critical estimate]\label{thm:hyperbolic-main}
Let $L$ be the operator in \eqref{eq:hyperbolic-normal}. For every dyadic $N\ge2$,
\begin{equation}
 \left\|e^{itL}P_{\le N}\right\|_{L^2(\R\times\T^2)\to L^{10/3}([0,1]\times \R\times\T^2)}
 \eqsim N^{1/5}.
 \label{eq:hyperbolic-main}
\end{equation}
\end{theorem}

The upper bound holds for every initial datum. The reverse inequality means that, for each $N$, we can choose a normalized datum whose spacetime norm is at least $cN^{1/5}$. The constants in both directions are independent of $N$.

\begin{remark}\label{rem:symmetry}
We may fix the signs as in \eqref{eq:hyperbolic-normal}. Indeed, let $(Rf)(x,y_1,y_2)\coloneqq f(x,y_2,y_1)$ and let $L'\coloneqq\partial_x^2-\partial_{y_1}^2+\partial_{y_2}^2$. The map $R$ is an isometry on every spatial $L^p$ space and on the corresponding spacetime spaces. It commutes with the radial cutoff. On Fourier modes we have
\[
 R^{-1}LR=L'\qquad\text{and}\qquad
 e^{itL'}P_{\le N}=R^{-1}e^{itL}P_{\le N}R.
\]
Taking operator norms proves that the two orderings give identical estimates, including the lower bounds and the time-scale characterization.
\end{remark}

The proof of \Thmref{thm:hyperbolic-main} keeps track of the interval length. Together with the local estimate in \Secref{sec:time}, it gives the following conclusions.

\begin{corollary}\label{cor:hyp-time}
For the operator $L$ in \eqref{eq:hyperbolic-normal}, we have
\begin{align}
 \mathcal S(N,T)&\eqsim (N^2T)^{3/10}
 &&\text{if }0<T\le N^{-2},\label{eq:cor-hyp-very-short}\\
 \mathcal S(N,T)&\eqsim 1
 &&\text{if }N^{-2}\le T\le N^{-1},\label{eq:cor-hyp-short}\\
 \mathcal S(N,T)&\eqsim(NT)^{1/5}
 &&\text{if }T\ge1.\label{eq:cor-hyp-long}
\end{align}
In the intermediate regime $N^{-1}<T<1$, we have
\begin{equation}
 (NT)^{1/5}\lesssim\mathcal S(N,T)
 \lesssim\min\{(NT)^{3/10},T^{1/10}N^{1/5}\}.
 \label{eq:cor-hyp-intermediate}
\end{equation}
The sharp dependence on $T$ in this intermediate regime remains open.

For each fixed $T_0>0$, the operator norm on $[0,T_0]$ has order $N^{1/5}$,
with constants allowed to depend on $T_0$. Moreover, a family
$T_N\in(0,1]$ admits a bound uniform in $N$ precisely when
\begin{equation}
 \sup_N\mathcal S(N,T_N)<\infty
 \quad\Longleftrightarrow\quad\sup_NNT_N<\infty.
 \label{eq:cor-hyp-lossless}
\end{equation}
\end{corollary}

\begin{remark}\label{rem:contribution}
The sharp long-time behavior in \Corref{cor:hyp-time} comes from a
balance between global cylinder dispersion and periodic null concentration.
For the upper bound, the global elliptic cylinder estimate is lifted to
Hilbert-valued data in a way that preserves its $\ell^8$ time summation. After
Bernstein in the remaining periodic direction and interpolation with $L^2$,
this yields the factor $(NT)^{1/5}$.  For the lower bound, a stationary
periodic null packet is paired with a Euclidean packet of frequency width $T^{-1/2}$,
which attains the same order for $T\ge1$ and also gives the obstruction to
extending the uniform lossless scale beyond $N^{-1}$.

The global time summability is essential.  Applying a sharp unit-time estimate
independently on consecutive intervals would only produce the larger time
factor $T^{3/10}$.  For comparison, the unit-time upper bound can also be
recovered from the compact hyperbolic estimate of Liu--Zheng
\cite{LiuZheng2025} via \Propref{prop:transfer}. The cylinder argument
used here contains the additional global information needed for the sharp
$T^{1/5}$ growth.
\end{remark}

\Thmref{thm:hyperbolic-main} addresses the critical estimate for the
specific hyperbolic model \eqref{eq:hyperbolic-normal} within the
higher-dimensional waveguide question discussed in \cite[Section~5]{DFZ2025}.
The global hyperbolic cylinder estimate proposed in
\cite[Remark~1.2]{DFZ2025} concerns a different flow and is not used here.

\begin{theorem}[Elliptic comparison]\label{thm:elliptic-main}
Suppose that $a_1,a_2,a_3$ have the same sign. For every $\varepsilon>0$ and every dyadic $N\ge2$,
\begin{equation}
 \left\|e^{itL_a}P_{\le N}f\right\|_{L^{10/3}([0,1]\times \R\times\T^2)}
 \lesssim_{a,\varepsilon}N^\varepsilon\|f\|_{L^2(\R\times\T^2)}.
 \label{eq:elliptic-main}
\end{equation}
\end{theorem}

\begin{remark}\label{rem:elliptic-known}
The estimate in \Thmref{thm:elliptic-main} is not new.  Up to the
standard rescaling to the corresponding rectangular torus, it is contained in
Barron's general endpoint Strichartz theory for semiperiodic Schr\"odinger
equations. See \cite[Theorem~1]{Barron2021}, specialized to
$\R\times\T^2$ at the critical exponent $p=10/3$.  We include the theorem
because the proof given below is different: it is obtained by a
Bloch--Floquet transfer from compact Strichartz estimates.  This formulation
also provides a convenient comparison with the hyperbolic problem and with the
short-time elliptic estimates considered later.
\end{remark}

The same transfer principle also gives a lossless estimate on a longer
frequency-dependent time window in the elliptic case. Here $\mathcal S_a(N,T)$
denotes the same operator norm as $\mathcal S(N,T)$, with $L$ replaced by $L_a$.

\begin{corollary}[Elliptic short-window estimate]\label{cor:elliptic-short}
Suppose that $a_1,a_2,a_3$ have the same sign and $0<\eta<1/5$. Then
\begin{equation}
 \mathcal S_a(N,T)\eqsim_{a,\eta}\min\{(N^2T)^{3/10},1\}
 \qquad\text{if }0<T\le N^{-4/5-\eta}.
 \label{eq:elliptic-short}
\end{equation}
\end{corollary}

\subsection{Proof strategy and organization}\label{sub:route}

\medskip
\noindent\emph{The hyperbolic critical estimate.}
\Secref{sec:hyperbolic-proof} proves
\Thmref{thm:hyperbolic-main}.  For the upper bound, we expand in the
negative periodic direction and reduce the problem to the global endpoint
estimate on the elliptic cylinder $\R\times\T$, with an $\ell^2$-valued
coefficient sequence.  Periodic Bernstein in the remaining variable and
interpolation with spacetime $L^2$ then give the critical $L^{10/3}$ bound.
The lower bound is furnished by a packet supported on the exact periodic null
line.  The Hilbert-valued cylinder estimate and the null-packet construction
are stated at the points where they first enter the proof and are established
immediately afterward.

\medskip
\noindent\emph{Dependence on the time interval.}
\Secref{sec:time} keeps the interval length explicit.  A local Bernstein
argument and a matching coherent packet settle the shortest regime. The
remaining short-time upper bound follows from a $TT^*$ kernel estimate.
Combining these local bounds with the global cylinder summability proves
\Corref{cor:hyp-time}, including the sharp law for $T\ge1$ and the
characterization of frequency-dependent lossless families.  The sharp
dependence on $T$ in the intermediate regime $N^{-1}<T<1$ remains open.

\medskip
\noindent\emph{The elliptic comparison.}
\Secref{sec:elliptic-proof} first obtains
\Thmref{thm:elliptic-main} from a compact estimate through a
Bloch--Floquet transfer and then proves the transfer principle itself.
\Corref{cor:elliptic-short} follows by applying Quinn's lossless critical
estimate on rectangular tori \cite[Theorem~1.1]{Quinn2026} after a fixed
coefficient rescaling and transferring the resulting bound back to the
waveguide.  This arrangement keeps the known compact inputs separate from
the arguments specific to the semiperiodic model.

\section{Proof of \Thmref{thm:hyperbolic-main}}\label{sec:hyperbolic-proof}

\subsection{Notation}\label{sub:notation}

\begin{itemize}
\item We write $z\coloneqq(x,y_1,y_2)\in\R\times\T^2$ and $\zeta\coloneqq(\xi,k_1,k_2)\in\R\times\Z^2$. Torus integrals use Lebesgue measure on $[0,2\pi)$. Frequency integrals use Lebesgue measure in $\xi$ and counting measure in $k$.
\item Our Fourier transform is
\[
 \widehat f(\xi,k)\coloneqq(2\pi)^{-3/2}
 \int_\R\int_{\T^2}f(x,y)e^{-i(x\xi+y\cdot k)}\,\dd y\dd x.
\]
Fourier inversion has the same prefactor, and Plancherel gives
\[
 \|f\|_{L^2(\R\times\T^2)}^2
 =\sum_{k\in\Z^2}\int_\R|\widehat f(\xi,k)|^2\,\dd\xi.
\]
Thus $e^{itL_a}$ has multiplier $e^{-itQ_a(\xi,k)}$.
\item The frequency $N\ge2$ is dyadic. We abbreviate the norm on a variable time interval by
\[
 \mathcal S_a(N,T)\coloneqq
 \left\|e^{itL_a}P_{\le N}\right\|_{L^2(\R\times\T^2)\to L^{10/3}([0,T]\times\R\times\T^2)}.
\]
For $a=(1,1,-1)$, this agrees with the notation $\mathcal S(N,T)$ introduced
in \Subsecref{sec:main-results}.
\item We write $A\lesssim B$ when $A\le CB$ with $C$ independent of $N,T$, and $A\eqsim B$ when both inequalities hold. Subscripts record additional dependence. Constants may depend on the fixed cutoffs. Within a proof, a norm with no indicated domain is taken over the space already specified there.
\end{itemize}

\subsection{The upper and lower bounds}\label{sub:main-proof}

We seek the two bounds in \Thmref{thm:hyperbolic-main}. For the upper bound, we retain the positive cylinder and regard the remaining periodic coordinate as a Hilbert-valued variable. The precise estimate we need is the following.

\begin{lemma}\label{lem:hilbert}
Let $\mathcal H$ be a Hilbert space, $I_\gamma\coloneqq[\gamma,\gamma+1]$, and $V(t)\coloneqq e^{it\Delta_{\R\times\T}}$. For $F\in L^2(\R\times\T,\mathcal H)$,
\begin{equation}
 \left(\sum_\gamma\|V(t)F\|_{L^4(I_\gamma\times\R\times\T,\mathcal H)}^8\right)^{1/8}
 \lesssim\|F\|_{L^2(\R\times\T,\mathcal H)}.
 \label{eq:hilbert-cylinder}
\end{equation}
\end{lemma}

For the lower bound, we need the periodic factor to remain stationary while the continuous factor spreads on the prescribed time scale.

\begin{lemma}\label{lem:null-packet}
For every dyadic $N\ge2$ and every $T\ge N^{-2}$, there is $f_{N,T}\in L^2(\R\times\T^2)$ such that $\|f_{N,T}\|_2=1$, $P_{\le N}f_{N,T}=f_{N,T}$, and
\[
 \|e^{itL}f_{N,T}\|_{L^{10/3}([0,T]\times\R\times\T^2)}\gtrsim(NT)^{1/5}.
\]
\end{lemma}

\begin{proof}[Proof of \Thmref{thm:hyperbolic-main}]
We now expand in the negative periodic direction,
\[
 P_{\le N}f(x,y_1,y_2)
 =(2\pi)^{-1/2}\sum_{|k|\le2N}f_k(x,y_1)e^{iky_2}.
\]
Then
\[
 u(t,x,y)\coloneqq e^{itL}P_{\le N}f(x,y)
 =(2\pi)^{-1/2}\sum_{|k|\le2N}e^{itk^2}V(t)f_k(x,y_1)e^{iky_2}.
\]
By periodic Bernstein and Parseval,
\begin{align*}
 \|u(t,x,y_1,\cdot)\|_{L^4(\T)}
 &\lesssim N^{1/4}\|u(t,x,y_1,\cdot)\|_{L^2(\T)}\\
 &=N^{1/4}\left(\sum_{|k|\le2N}|V(t)f_k(x,y_1)|^2\right)^{1/2}.
\end{align*}
We apply \eqref{eq:hilbert-cylinder} with $\mathcal H=\ell^2(\Z)$ and obtain
\begin{align}
 \left(\sum_\gamma\|u\|_{L^4(I_\gamma\times\R\times\T^2)}^8\right)^{1/8}
 &\lesssim N^{1/4}\left(\sum_\gamma
 \|(V(t)f_k)_k\|_{L^4(I_\gamma\times\R\times\T,\ell^2)}^8\right)^{1/8}\notag\\
 &\lesssim N^{1/4}\left(\sum_k\|f_k\|_2^2\right)^{1/2}
 \le N^{1/4}\|f\|_2.
 \label{eq:waveguide-L4-sequence}
\end{align}

For $T\ge1$, let $J_T\coloneqq\lceil T\rceil$. H\"older's inequality gives
\begin{align*}
 \|u\|_{L^4([0,T]\times\R\times\T^2)}^4
 &\le\sum_{\gamma=0}^{J_T-1}\|u\|_{L^4(I_\gamma\times\R\times\T^2)}^4\\
 &\le J_T^{1/2}\left(\sum_{\gamma=0}^{J_T-1}
 \|u\|_{L^4(I_\gamma\times\R\times\T^2)}^8\right)^{1/2}
 \lesssim T^{1/2}N\|f\|_2^4.
\end{align*}
Interpolating with $\|u\|_{L^2([0,T]\times\R\times\T^2)}\le T^{1/2}\|f\|_2$, we find
\begin{align}
 \|u\|_{L^{10/3}([0,T]\times\R\times\T^2)}
 &\le\|u\|_{L^4([0,T]\times\R\times\T^2)}^{4/5}
 \|u\|_{L^2([0,T]\times\R\times\T^2)}^{1/5}\notag\\
 &\lesssim(NT)^{1/5}\|f\|_2.
 \label{eq:hyp-upper}
\end{align}
In particular, $T=1$ gives the required upper bound.

For the reverse inequality, we apply \Lemref{lem:null-packet} with $T=1$ and obtain
\[
 \|e^{itL}P_{\le N}\|_{L^2(\R\times\T^2)\to L^{10/3}([0,1]\times\R\times\T^2)}
 \ge \|e^{itL}f_{N,1}\|_{L^{10/3}([0,1]\times\R\times\T^2)}
 \gtrsim N^{1/5}.
\]
This proves the theorem. The calculation above also establishes the upper bound for every $T\ge1$.
\end{proof}

\subsection{Proofs of the two lemmas}\label{sub:main-lemmas}

\begin{proof}[Proof of \Lemref{lem:hilbert}]
We use the global cylinder estimate of Barron--Christ--Pausader \cite[Theorem~1.1]{BCP2021}.
Let $I_\gamma\coloneqq[\gamma,\gamma+1]$. For every $g\in L^2(\R\times\T)$,
\begin{equation}
 \left(\sum_{\gamma\in\Z}
 \|e^{it\Delta_{\R\times\T}}g\|_{L^4(I_\gamma\times\R\times\T)}^8\right)^{1/8}
 \lesssim\|g\|_2.
 \label{eq:BCP-benchmark}
\end{equation}

Let $\mathcal H$ be a Hilbert space. Let $V(t)\coloneqq e^{it\Delta_{\R\times\T}}$ and let $(e_j)_{j=1}^J$ be orthonormal. For
\[
 F\coloneqq\sum_{j=1}^Jf_je_j
 \qquad\text{and}\qquad
 a_{\gamma,j}\coloneqq\|V(t)f_j\|_{L^4(I_\gamma\times\R\times\T)},
\]
orthogonality and Minkowski give
\begin{align*}
 \|V(t)F\|_{L^4(I_\gamma\times\R\times\T,\mathcal H)}^2
 &=\left\|\sum_{j=1}^J|V(t)f_j|^2\right\|_{L^2(I_\gamma\times\R\times\T)}\\
 &\le\sum_{j=1}^J\bigl\||V(t)f_j|^2\bigr\|_{L^2(I_\gamma\times\R\times\T)}
 =\sum_{j=1}^Ja_{\gamma,j}^2.
\end{align*}
Applying Minkowski once more and then \eqref{eq:BCP-benchmark}, we obtain
\begin{align*}
 \left(\sum_\gamma\|V(t)F\|_{L^4(I_\gamma\times\R\times\T,\mathcal H)}^8\right)^{1/4}
 &\le\left(\sum_\gamma\left(\sum_{j=1}^Ja_{\gamma,j}^2\right)^4\right)^{1/4}\\
 &\le\sum_{j=1}^J\left(\sum_\gamma a_{\gamma,j}^8\right)^{1/4}\\
 &\lesssim\sum_{j=1}^J\|f_j\|_2^2
 =\|F\|_{L^2(\R\times\T,\mathcal H)}^2.
\end{align*}
The constant is independent of $J$.  Finite-rank simple functions are dense in
the Bochner space $L^2(\R\times\T,\mathcal H)$ (after restricting to the
separable closed subspace generated by the essential range of $F$).  Taking
square roots and passing to the limit therefore proves the lemma for every
$F\in L^2(\R\times\T,\mathcal H)$.

\end{proof}

\begin{proof}[Proof of \Lemref{lem:null-packet}]
We choose frequencies on the periodic null line. Define
\[
 J_N\coloneqq\{j\in\Z\mid N/2\le j\le5N/8\},
 \qquad\text{and}\qquad M_N\coloneqq\#J_N\eqsim N.
\]
The function
\[
 H_N(y_1,y_2)\coloneqq\frac{1}{2\pi\sqrt{M_N}}
 \sum_{j\in J_N}e^{ij(y_1+y_2)}
\]
satisfies
\begin{align*}
 \|H_N\|_{L^2(\T^2)}^2
 &=\frac{1}{4\pi^2M_N}\sum_{j\in J_N}4\pi^2=1,\\
 e^{it(\partial_{y_1}^2-\partial_{y_2}^2)}H_N
 &=\frac{1}{2\pi\sqrt{M_N}}
 \sum_{j\in J_N}e^{it(-j^2+j^2)}e^{ij(y_1+y_2)}=H_N.
\end{align*}
Let $\theta\in[-\pi,\pi]$ represent $y_1+y_2$ modulo $2\pi$, and define
\[
 \Omega_N\coloneqq\{(y_1,y_2)\in\T^2\mid |\theta|\le(100M_N)^{-1}\}.
\]
On this strip, removing the phase of the first summand gives
\begin{align*}
 \left|\sum_{j\in J_N}e^{ij\theta}\right|
 &=\left|\sum_{m=0}^{M_N-1}e^{im\theta}\right|
 \ge\sum_{m=0}^{M_N-1}\cos(m\theta)
 \ge M_N\cos(1/100),\\
 |\Omega_N|&=2\pi\frac{2}{100M_N}\eqsim M_N^{-1}.
\end{align*}
Consequently,
\begin{align}
 \|H_N\|_{L^{10/3}(\T^2)}
 &\ge\left(\int_{\Omega_N}|H_N(y)|^{10/3}\,\dd y\right)^{3/10}\notag\\
 &\gtrsim M_N^{1/2}|\Omega_N|^{3/10}
 \eqsim N^{1/5}.
 \label{eq:null-dirichlet}
\end{align}

We choose a nonnegative nonzero $\widehat g\in C_c^\infty((-1,1))$ with $\|g\|_2=1$. For $T\ge N^{-2}$, define
\[
 r\coloneqq10^{-2}T^{-1/2}
 \qquad\text{and}\qquad g_r(x)\coloneqq r^{1/2}g(rx).
\]
Then
\[
 \widehat g_r(\xi)=r^{-1/2}\widehat g(\xi/r),
 \qquad \supp\widehat g_r\subseteq(-r,r).
\]
Fourier inversion gives
\[
 e^{it\partial_x^2}g_r(x)=\frac{r^{1/2}}{\sqrt{2\pi}}
 \int_\R e^{i(rx\eta-tr^2\eta^2)}\widehat g(\eta)\,\dd\eta.
\]
On $0\le t\le T$, $|x|\le(100r)^{-1}$, and $|\eta|\le1$, positivity of $\widehat g$ gives
\begin{align*}
 |rx\eta-tr^2\eta^2|&\le10^{-2}+10^{-4}<\pi/3,\\
 |e^{it\partial_x^2}g_r(x)|
 &\ge\frac{r^{1/2}}{2\sqrt{2\pi}}\int_\R\widehat g(\eta)\,\dd\eta
 \gtrsim r^{1/2},\\
 \|e^{it\partial_x^2}g_r\|_{L^{10/3}([0,T]\times\R)}^{10/3}
 &\ge\int_0^T\int_{|x|\le(100r)^{-1}}
 |e^{it\partial_x^2}g_r(x)|^{10/3}\,\dd x\dd t\\
 &\gtrsim Tr^{5/3}r^{-1}\eqsim T^{2/3}.
\end{align*}
For $f_{N,T}(x,y)\coloneqq g_r(x)H_N(y)$, we have $\|f_{N,T}\|_2=1$.
Since $T\ge N^{-2}$, we have $r\le10^{-2}N$.  Thus, on the Fourier support
of $f_{N,T}$,
\[
 \xi^2+2j^2\le(10^{-4}+25/32)N^2<N^2,
\]
and hence $P_{\le N}f_{N,T}=f_{N,T}$. The periodic factor is stationary, giving
\begin{align}
 \mathcal S(N,T)
 &\ge\|e^{itL}f_{N,T}\|_{L^{10/3}([0,T]\times\R\times\T^2)}\notag\\
 &=\|e^{it\partial_x^2}g_r\|_{L^{10/3}([0,T]\times\R)}
 \|H_N\|_{L^{10/3}(\T^2)}
 \gtrsim (NT)^{1/5}.
 \label{eq:null-packet}
\end{align}
The same computation gives the asserted norm of $e^{itL}f_{N,T}$.
\end{proof}

\section{Time intervals and proof of \Corref{cor:hyp-time}}\label{sec:time}

\subsection{Derivation of the corollary}\label{sub:time-corollary}

The proof of \Thmref{thm:hyperbolic-main} already gives the long-time upper bound and the null-packet lower bound. To settle the short intervals and the uniform families in \Corref{cor:hyp-time}, we need the following estimate, which applies to both signatures.

\begin{proposition}\label{prop:short}
For every fixed $a\in(\R\setminus\{0\})^3$,
\begin{equation}
 \mathcal S_a(N,T)\eqsim_a\min\{(N^2T)^{3/10},1\}
 \qquad\text{if }0<T\le N^{-1}.
 \label{eq:short}
\end{equation}
The upper bound by $(N^2T)^{3/10}$ holds for all $T>0$. The lower bound by the minimum holds for $0<T\le1$.
\end{proposition}

\begin{proof}[Proof of \Corref{cor:hyp-time}]
\Propref{prop:short} gives
\eqref{eq:cor-hyp-very-short} and \eqref{eq:cor-hyp-short} by separating the
two alternatives in the minimum in \eqref{eq:short}.  The upper bound
\eqref{eq:hyp-upper} and the lower bound \eqref{eq:null-packet} give
\eqref{eq:cor-hyp-long}.

It remains to verify the bounds in the intermediate regime.  Since $NT>1$,
\eqref{eq:null-packet} gives the lower bound in
\eqref{eq:cor-hyp-intermediate}.  Partitioning $[0,T]$ into at most $2NT$
intervals of length at most $N^{-1}$ and summing \eqref{eq:local-upper} in
$L^{10/3}$ gives
\[
 \mathcal S(N,T)\lesssim (NT)^{3/10}.
\]
On the other hand, \eqref{eq:waveguide-L4-sequence} on $[0,T]\subseteq[0,1]$
and interpolation with the spacetime $L^2$ estimate give
\begin{align*}
 \|e^{itL}P_{\le N}f\|_{L^{10/3}([0,T]\times\R\times\T^2)}
 &\le\|e^{itL}P_{\le N}f\|_{L^4([0,T]\times\R\times\T^2)}^{4/5}
 \|e^{itL}P_{\le N}f\|_{L^2([0,T]\times\R\times\T^2)}^{1/5}\\
 &\lesssim T^{1/10}N^{1/5}\|f\|_2.
\end{align*}
Taking the better of these two bounds proves
\eqref{eq:cor-hyp-intermediate}.

Suppose $NT_N\le K$. We partition $[0,T_N]$ into $J_N'\le1+K$ intervals $[s_j,s_j+\ell_j]$ with $\ell_j\le N^{-1}$. Time translation and \eqref{eq:local-upper} give
\begin{align*}
 \|e^{itL}P_{\le N}f\|_{L^{10/3}([0,T_N]\times\R\times\T^2)}^{10/3}
 &=\sum_{j=1}^{J_N'}\int_0^{\ell_j}
 \|e^{i\tau L}P_{\le N}(e^{is_jL}f)\|_{L^{10/3}(\R\times\T^2)}^{10/3}\dd\tau\\
 &\lesssim\sum_{j=1}^{J_N'}\|e^{is_jL}f\|_2^{10/3}
 \le(1+K)\|f\|_2^{10/3}.
\end{align*}
Conversely, if $\mathcal S(N,T_N)\le M_0$, then \eqref{eq:null-packet} implies
\[
 NT_N\le
 \begin{cases}
 C M_0^5,&T_N\ge N^{-2},\\
 N^{-1},&T_N<N^{-2}.
 \end{cases}
\]
This proves \eqref{eq:cor-hyp-lossless}.

For fixed $T_0\ge1$, we use \eqref{eq:cor-hyp-long}. For $0<T_0<1$ and $N^{-2}\le T_0$, we have
\[
 T_0^{1/5}N^{1/5}\lesssim\mathcal S(N,T_0)
 \le\mathcal S(N,1)\lesssim N^{1/5}.
\]
The finitely many remaining frequencies satisfy \eqref{eq:coherent-lower}, and we conclude that $\mathcal S(N,T_0)\eqsim_{T_0}N^{1/5}$.
\end{proof}

\subsection{Proof of the short-time estimate}\label{sub:local}

We now prove \Propref{prop:short}. Bernstein handles the shortest intervals. To reach length $N^{-1}$, we factor the truncated kernel and use its dispersive bound in $TT^*$.

We choose a real $\rho\in C_c^\infty(\R)$ with $0\le\rho\le1$, equal to one on $[-2,2]$ and supported in $[-4,4]$, and define
\[
 \Pi_N\coloneqq\rho(D_x/N)\rho(D_{y_1}/N)\rho(D_{y_2}/N).
\]
We will use
\begin{equation}
 \Pi_NP_{\le N}=P_{\le N}.
 \label{eq:Pi-identity}
\end{equation}
Indeed, we write
\[
 p_N(\xi,k)\coloneqq\chi\bigl(N^{-1}\sqrt{\xi^2+|k|^2}\bigr),
\]
and
\[
 m_N(\xi,k)\coloneqq\rho(\xi/N)\rho(k_1/N)\rho(k_2/N).
\]
The choice of $\rho$ gives
\begin{align*}
 (\xi,k)\in\supp p_N
 &\ \Longrightarrow\ \max\{|\xi|,|k_1|,|k_2|\}\le2N
 \ \Longrightarrow\ m_N(\xi,k)=1,\\
 \widehat{\Pi_NP_{\le N}f}
 &=m_Np_N\widehat f
 =p_N\widehat f
 =\widehat{P_{\le N}f}.
\end{align*}
By Plancherel, we obtain \eqref{eq:Pi-identity} and also
\[
 \Pi_N^2P_{\le N}=P_{\le N}=P_{\le N}\Pi_N.
\]

The kernel bound needed in this argument is the following. We defer its proof until after \Propref{prop:short}.

\begin{lemma}\label{lem:kernel}
For $b\ne0$ and $0<|t|\le1$,
\begin{align}
 \sup_v\left|\int_\R e^{i(v\xi-bt\xi^2)}\rho(\xi/N)^2\,\dd\xi\right|
 &\lesssim_b\min\{N,|t|^{-1/2}\},\label{eq:factor-cont}\\
 \sup_y\left|\sum_{k\in\Z}e^{i(ky-btk^2)}\rho(k/N)^2\right|
 &\lesssim_b\min\{N,|t|^{-1/2}(1+N|t|)\}.\label{eq:factor-per}
\end{align}
The kernel $K_N$ of $e^{itL_a}\Pi_N^2$ satisfies
\begin{equation}
 \|K_N(t)\|_\infty\lesssim_a\min\{N^3,|t|^{-3/2}\}
 \qquad\text{for }0<|t|\le N^{-1}.
 \label{eq:local-kernel}
\end{equation}
\end{lemma}

\begin{proof}[Proof of \Propref{prop:short}]
We first use Bernstein and unitarity to obtain
\begin{align*}
 \|e^{itL_a}P_{\le N}f\|_{L^{10/3}(\R\times\T^2)}
 &\lesssim N^{3/5}\|e^{itL_a}P_{\le N}f\|_2
 \le N^{3/5}\|f\|_2,\\
 \|e^{itL_a}P_{\le N}f\|_{L^{10/3}([0,T]\times\R\times\T^2)}
 &\lesssim T^{3/10}N^{3/5}\|f\|_2.
\end{align*}
Hence
\begin{equation}
 \mathcal S_a(N,T)\lesssim(N^2T)^{3/10}.
 \label{eq:bernstein}
\end{equation}

We next take $I\coloneqq[0,T]$, $T\le N^{-1}$, and define
\[
 Sf(t)\coloneqq\mathbf1_I(t)e^{itL_a}\Pi_Nf.
\]
For $F\in C_c^\infty(\R\times\R\times\T^2)$, the Fourier multipliers commute, and
\begin{align*}
 S^*F&=\int_I\Pi_Ne^{-isL_a}F(s)\,\dd s,\\
 SS^*F(t)&=\mathbf1_I(t)\int_Ie^{i(t-s)L_a}\Pi_N^2F(s)\,\dd s.
\end{align*}
Interpolating \eqref{eq:local-kernel} with the $L^2$ estimate gives
\[
 \|e^{ihL_a}\Pi_N^2F\|_{L^{10/3}(\R\times\T^2)}
 \lesssim_a |h|^{-3/5}\|F\|_{L^{10/7}(\R\times\T^2)}
 \qquad(0<|h|\le N^{-1}).
\]
We apply the Hardy--Littlewood--Sobolev inequality
\begin{equation}
 \left\|\int_\R|t-s|^{-3/5}b(s)\,\dd s\right\|_{L^{10/3}_t(\R)}
 \lesssim\|b\|_{L^{10/7}(\R)}
 \label{eq:hls}
\end{equation}
to $b(s)\coloneqq\mathbf1_I(s)\|F(s)\|_{L^{10/7}(\R\times\T^2)}$. Together with H\"older's inequality, duality, and \eqref{eq:Pi-identity}, this gives
\begin{align*}
 \|SS^*F\|_{L^{10/3}(\R\times\R\times\T^2)}
 &\lesssim_a\left\|\int_I|t-s|^{-3/5}
 \|F(s)\|_{L^{10/7}(\R\times\T^2)}\,\dd s\right\|_{L^{10/3}_t(I)}\\
 &\lesssim_a\|b\|_{L^{10/7}(\R)}
 =\|F\|_{L^{10/7}(I\times\R\times\T^2)},\\
 \|S^*F\|_2^2
 &=\langle SS^*F,F\rangle
 \le\|SS^*F\|_{10/3}\|F\|_{10/7}
 \lesssim_a\|F\|_{10/7}^2,\\
 \|e^{itL_a}P_{\le N}f\|_{L^{10/3}(I\times\R\times\T^2)}
 &=\|S(P_{\le N}f)\|_{10/3}
 \lesssim_a\|P_{\le N}f\|_2
 \le\|f\|_2.
\end{align*}
Thus
\begin{equation}
 \mathcal S_a(N,T)\lesssim_a1\qquad(0<T\le N^{-1}).
 \label{eq:local-upper}
\end{equation}

For the lower bound, we take $N\ge16$, $M\coloneqq N/16$, and a nonnegative $\phi\in C_c^\infty((-1,1))$ with $\|\phi\|_2=1$. Define
\[
 \widehat g_M(\xi)\coloneqq M^{-1/2}\phi(\xi/M),
 \qquad\text{and}\qquad
 h_M(y)\coloneqq\frac{1}{\sqrt{2\pi(2M+1)}}\sum_{k=-M}^Me^{iky}.
\]
For $F_M(x,y)\coloneqq g_M(x)h_M(y_1)h_M(y_2)$, we have
\[
 \|F_M\|_2=1,
 \qquad\supp\widehat F_M\subseteq\{|\zeta|\le\sqrt3M\},
 \qquad\text{and}\qquad P_{\le N}F_M=F_M.
\]
We choose $b_a>0$ small enough that $b_a(1+\max_j|a_j|)<\pi/3$. On
\[
 0\le t\le b_aM^{-2},
 \qquad |x|,|y_1|,|y_2|\le b_aM^{-1},
\]
the phases satisfy
\[
 |Mx\eta-a_1tM^2\eta^2|<\pi/3
 \quad\text{and}\quad
 |ky_j-a_{j+1}tk^2|<\pi/3
 \qquad(|\eta|\le1,\ |k|\le M).
\]
Taking real parts, we obtain
\begin{align*}
 \left|e^{ita_1\partial_x^2}g_M(x)\right|
 &=\frac{M^{1/2}}{\sqrt{2\pi}}
 \left|\int_{-1}^1e^{i(Mx\eta-a_1tM^2\eta^2)}\phi(\eta)\,\dd\eta\right|
 \gtrsim M^{1/2},\\
 \left|e^{ita_{j+1}\partial_{y_j}^2}h_M(y_j)\right|
 &=\frac{\left|\sum_{k=-M}^Me^{i(ky_j-a_{j+1}tk^2)}\right|}
 {\sqrt{2\pi(2M+1)}}
 \gtrsim M^{1/2},\\
 |e^{itL_a}F_M(x,y)|&\gtrsim M^{3/2}.
\end{align*}
With $\tau\coloneqq\min\{T,b_aM^{-2}\}$, it follows that
\begin{align*}
 \|e^{itL_a}P_{\le N}F_M\|_{L^{10/3}([0,T]\times\R\times\T^2)}^{10/3}
 &\ge\int_0^\tau\int_{|x|,|y_1|,|y_2|\le b_aM^{-1}}
 |e^{itL_a}F_M|^{10/3}\,\dd x\dd y\dd t\\
 &\gtrsim_a M^5\tau M^{-3}
 \gtrsim_a\min\{N^2T,1\}.
\end{align*}
We have proved
\begin{equation}
 \mathcal S_a(N,T)\gtrsim_a\min\{(N^2T)^{3/10},1\}
 \qquad(0<T\le1).
 \label{eq:coherent-lower}
\end{equation}
For the remaining frequencies $N\in\{2,4,8\}$, choose
$0<\delta_a<1$ such that
$\delta_a+|a_1|\delta_a^2<\pi/3$, and choose a nonnegative
$\widehat g_a\in C_c^\infty((-\delta_a,\delta_a))$ with $\|g_a\|_2=1$.
Set
\[
 F_a(x,y_1,y_2)\coloneqq(2\pi)^{-1}g_a(x).
\]
Then $\|F_a\|_2=1$, $P_{\le N}F_a=F_a$, and for
$0\le t\le1$ and $|x|\le1$, taking real parts in the Fourier inversion
formula gives $|e^{itL_a}F_a(x,y)|\gtrsim_a1$.  Therefore
\[
 \mathcal S_a(N,T)
 \ge\|e^{itL_a}F_a\|_{L^{10/3}([0,T]\times[-1,1]\times\T^2)}
 \gtrsim_a T^{3/10}
 \gtrsim_a\min\{(N^2T)^{3/10},1\}
 \qquad(0<T\le1).
\]
Combining this with \eqref{eq:bernstein} and \eqref{eq:local-upper} completes the proof.
\end{proof}

\subsection{Proof of the kernel bound}\label{sub:kernel}

\begin{proof}[Proof of \Lemref{lem:kernel}]
We apply van der Corput to the continuous factor. For
$\phi(\xi)\coloneqq v\xi-bt\xi^2$,
\begin{align*}
 |\phi''(\xi)|&=2|bt|,\\
 \|\rho(\cdot/N)^2\|_\infty
 +\int_\R\left|\partial_\xi[\rho(\xi/N)^2]\right|\dd\xi
 &=\|\rho^2\|_\infty+\int_\R|\partial_\eta[\rho(\eta)^2]|\dd\eta
 \lesssim1,\\
 \left|\int_\R e^{i\phi(\xi)}\rho(\xi/N)^2\,\dd\xi\right|
 &\lesssim_b |t|^{-1/2},
\end{align*}
while integration of the amplitude gives the bound $CN$. This proves \eqref{eq:factor-cont}.

For $y\in[-\pi,\pi]$, Poisson summation gives
\[
 \sum_{k\in\Z}e^{i(ky-btk^2)}\rho(k/N)^2=\sum_{\ell\in\Z}I_\ell,
\]
where
\[
 I_\ell\coloneqq\int_\R e^{i((y-2\pi\ell)\xi-bt\xi^2)}\rho(\xi/N)^2\,\dd\xi.
\]
We estimate the terms with $|\ell|\le C_b(1+N|t|)$ by \eqref{eq:factor-cont}, taking $C_b$ sufficiently large. For the remaining terms, we change variables $\xi=N\eta$ and define
\[
 a(\eta)\coloneqq\rho(\eta)^2
 \qquad\text{and}\qquad
 \Phi_\ell(\eta)\coloneqq N(y-2\pi\ell)\eta-btN^2\eta^2.
\]
On $|\eta|\le4$, we have
\begin{align*}
 |\Phi_\ell'(\eta)|
 &=|N(y-2\pi\ell)-2btN^2\eta|
 \gtrsim_b N|\ell|,\\
 |\Phi_\ell''(\eta)|&=2|bt|N^2\lesssim_b N|\ell|,
 \qquad\text{and}\qquad \Phi_\ell'''=0.
\end{align*}
We integrate by parts twice, using the compact support of $a$ to remove the boundary terms, and obtain
\begin{align*}
 |I_\ell|
 &=N\left|\int_\R e^{i\Phi_\ell(\eta)}
 \partial_\eta\left[\frac{1}{i\Phi_\ell'}
 \partial_\eta\left(\frac{a}{i\Phi_\ell'}\right)\right]\dd\eta\right|\\
 &\lesssim N\int_{-4}^4
 \left(\frac{|a''|}{|\Phi_\ell'|^2}
 +\frac{|a'||\Phi_\ell''|}{|\Phi_\ell'|^3}
 +\frac{|a||\Phi_\ell''|^2}{|\Phi_\ell'|^4}\right)\dd\eta\\
 &\lesssim_b\frac{1}{N|\ell|^2}
 \int_{-4}^4(|a''|+|a'|+|a|)\,\dd\eta
 \lesssim_b N^{-1}|\ell|^{-2}.
\end{align*}
Summing these estimates and \eqref{eq:factor-cont}, we obtain
\begin{align*}
 \left|\sum_{k\in\Z}e^{i(ky-btk^2)}\rho(k/N)^2\right|
 &\le\sum_{|\ell|\le C_b(1+N|t|)}|I_\ell|
 +\sum_{|\ell|>C_b(1+N|t|)}|I_\ell|\\
 &\lesssim_b(1+N|t|)|t|^{-1/2}+N^{-1}\sum_{\ell\ne0}|\ell|^{-2}\\
 &\lesssim_b(1+N|t|)|t|^{-1/2}.
\end{align*}
Counting the nonzero terms also gives $CN$, proving \eqref{eq:factor-per}. We multiply the three factors and use $|t|\le N^{-1}$ to obtain
\begin{align*}
 \|K_N(t)\|_\infty
 &\lesssim_a\min\{N,|t|^{-1/2}\}
 \prod_{j=2}^3\min\{N,|t|^{-1/2}(1+N|t|)\}\\
 &\lesssim_a\bigl(\min\{N,|t|^{-1/2}\}\bigr)^3
 =\min\{N^3,|t|^{-3/2}\}.
\end{align*}
\end{proof}

\section{Elliptic comparison estimates}\label{sec:elliptic-proof}

\subsection{Proof of \Thmref{thm:elliptic-main}}\label{sub:transfer}

Although \Thmref{thm:elliptic-main} is already contained in Barron's
general semiperiodic endpoint theory, we record the Bloch--Floquet proof used
here.  Its compact input is the following known estimate.

\begin{lemma}[{\cite[Theorem~2.4 and Remark~2.5]{BourgainDemeter2015}}]\label{lem:compact}
Let $a_j>0$ be fixed. If $g\in L^2(\T^3)$ has Fourier support in $\{|k|\le3N\}$, then, for every $\varepsilon>0$,
\begin{equation}
 \|e^{itL_a}g\|_{L^{10/3}([0,1]\times\T^3)}
 \lesssim_{a,\varepsilon}N^\varepsilon\|g\|_2.
 \label{eq:BD-benchmark}
\end{equation}
\end{lemma}

Here we use the $n=4$, $p=10/3$ case of the cited theorem. Remark~2.5 allows the fixed positive definite form $\sum_{j=1}^3a_jk_j^2$. Changing the torus normalization and replacing the frequency radius by $3N$ only changes the constant. We will use this estimate with the same constant on every Bloch fibre.

To apply this estimate on $\R\times\T^2$, we need a transfer that preserves the constant. The next proposition supplies precisely this step.

\begin{proposition}[Bloch transfer]\label{prop:transfer}
Let $2\le r<\infty$ and let $I$ be an interval. Suppose that
\begin{equation}
 \|e^{itL_a}g\|_{L^r(I\times\T^3)}\le C(N,I)\|g\|_2
 \label{eq:transfer-assumption}
\end{equation}
for all $g$ with Fourier support in $\{|\ell|\le3N\}$. Then every $f\in L^2(\R\times\T^2)$ with Fourier support in $\{|\zeta|\le2N\}$ satisfies
\begin{equation}
 \|e^{itL_a}f\|_{L^r(I\times\R\times\T^2)}\le C(N,I)\|f\|_2.
 \label{eq:transfer-conclusion}
\end{equation}
\end{proposition}

\begin{proof}[Proof of \Thmref{thm:elliptic-main}]
For $a_j>0$, \Lemref{lem:compact} and \Propref{prop:transfer} give
\begin{align*}
 \|e^{itL_a}P_{\le N}f\|_{L^{10/3}([0,1]\times\R\times\T^2)}
 &\lesssim_{a,\varepsilon}N^\varepsilon\|P_{\le N}f\|_2
 \le N^\varepsilon\|f\|_2.
\end{align*}
For $a_j<0$, we use the reality and evenness of the cutoff to write
\[
 e^{itL_a}P_{\le N}f
 =\overline{e^{itL_{-a}}P_{\le N}\overline f}.
\]
The positive-coefficient estimate now proves the assertion.
\end{proof}

\subsection{Proof of the Bloch transfer}\label{sub:transfer-proof}

We write the continuous frequency as $\xi=n+\theta$, where $n\in\Z$ and $0\le\theta<1$. This identifies the compact estimate used above on each fibre.

\begin{proof}[Proof of \Propref{prop:transfer}]
We first take $f$ to be Schwartz with the stated support. Let $B\coloneqq[0,2\pi)$. For $x\in B$ and $\theta\in[0,1)$, define
\begin{equation}
 g_\theta(x,y)\coloneqq e^{-ix\theta}
 \sum_{j\in\Z}e^{-2\pi ij\theta}f(x+2\pi j,y).
 \label{eq:bloch}
\end{equation}
The sum converges with all derivatives. We replace $j+1$ by $m$ to obtain
\begin{align*}
 g_\theta(x+2\pi,y)
 &=e^{-i(x+2\pi)\theta}\sum_{j\in\Z}
 e^{-2\pi ij\theta}f(x+2\pi(j+1),y)\\
 &=e^{-ix\theta}\sum_{m\in\Z}e^{-2\pi im\theta}f(x+2\pi m,y)
 =g_\theta(x,y).
\end{align*}
By Parseval in $\theta$,
\[
 \int_0^1|g_\theta(x,y)|^2\,\dd\theta
 =\sum_{j\in\Z}|f(x+2\pi j,y)|^2.
\]
Integrating over $B\times\T^2$ gives
\begin{align}
 \int_0^1\|g_\theta\|_{L^2(\T^3)}^2\,\dd\theta
 &=\int_B\int_{\T^2}\sum_{j\in\Z}|f(x+2\pi j,y)|^2\,\dd y\dd x\notag\\
 &=\int_\R\int_{\T^2}|f(x,y)|^2\,\dd y\dd x
 =\|f\|_{L^2(\R\times\T^2)}^2.
 \label{eq:bloch-isometry}
\end{align}
We change variables $v=x+2\pi j$ in the Fourier coefficient. Since $e^{2\pi ijn}=1$,
\begin{align}
 \widehat g_\theta(n,k)
 &=(2\pi)^{-3/2}\sum_{j\in\Z}\int_B\int_{\T^2}
 f(x+2\pi j,y)e^{-i(x(n+\theta)+y\cdot k+2\pi j\theta)}\,\dd y\dd x\notag\\
 &=(2\pi)^{-3/2}\int_\R\int_{\T^2}
 f(v,y)e^{-i(v(n+\theta)+y\cdot k)}\,\dd y\dd v
 =\widehat f(n+\theta,k).
 \label{eq:bloch-frequency}
\end{align}
Thus
\[
 \widehat g_\theta(n,k)\ne0
 \quad\Longrightarrow\quad
 |(n,k)|\le|(n+\theta,k)|+\theta\le2N+1\le3N.
\]

Let $u\coloneqq e^{itL_a}f$, and define $w_\theta$ from $u$ by \eqref{eq:bloch}. Expanding the quadratic phase and taking the Fourier series, we find
\begin{align}
 \widehat w_\theta(t,n,k)
 &=e^{-it[a_1(n+\theta)^2+a_2k_1^2+a_3k_2^2]}\widehat g_\theta(n,k)\notag\\
 &=e^{-ita_1\theta^2}e^{-2ita_1n\theta}
 e^{-it[a_1n^2+a_2k_1^2+a_3k_2^2]}\widehat g_\theta(n,k)\notag,\\
 w_\theta(t,x,y)&=e^{-ita_1\theta^2}
 (e^{itL_a}g_\theta)(x-2ta_1\theta,y).
 \label{eq:twist}
\end{align}
Translation preserves torus measure, and the scalar factor has modulus one. By \eqref{eq:transfer-assumption},
\[
 \|w_\theta\|_{L^r(I\times\T^3)}
 =\|e^{itL_a}g_\theta\|_{L^r(I\times\T^3)}
 \le C(N,I)\|g_\theta\|_2.
\]
For $r\ge2$, the pointwise inclusion $\ell^2\subseteq\ell^r$ gives
\[
 \left(\sum_j|u_j|^r\right)^{2/r}\le\sum_j|u_j|^2.
\]
We apply this inequality, Parseval in the Bloch parameter, and then Minkowski
to obtain
\begin{align*}
 \|u\|_{L^r(I\times\R\times\T^2)}^2
 &=\left(\int_I\int_B\int_{\T^2}
 \sum_j|u(t,x+2\pi j,y)|^r\,\dd y\dd x\dd t\right)^{2/r}\\
 &\le\left\|\sum_j|u(t,x+2\pi j,y)|^2\right\|_{L^{r/2}(I\times\T^3)}\\
 &=\left\|\int_0^1|w_\theta(t,x,y)|^2\,\dd\theta\right\|_{L^{r/2}(I\times\T^3)}\\
 &\le\int_0^1\|w_\theta\|_{L^r(I\times\T^3)}^2\,\dd\theta\\
 &\le C(N,I)^2\int_0^1\|g_\theta\|_2^2\,\dd\theta
 =C(N,I)^2\|f\|_2^2.
\end{align*}
Approximation by Schwartz data with the same support bound extends this estimate to $L^2$. Unitarity identifies the limit with $e^{itL_a}f$.
\end{proof}

\subsection{A short-window refinement}\label{sub:elliptic-short}

We now combine the Bloch--Floquet transfer with Quinn's lossless critical
estimate on rectangular tori \cite[Theorem~1.1]{Quinn2026}.  This avoids
reproducing the multiscale argument from \cite{Quinn2026}. Only the conversion
between the fixed positive coefficients and a rectangular torus is needed.

\begin{proof}[Proof of \Corref{cor:elliptic-short}]
By conjugation we may assume $a_j>0$.  In view of
\eqref{eq:bernstein}, \eqref{eq:coherent-lower}, and
\Propref{prop:transfer}, it is enough to prove the compact lossless
estimate
\begin{equation}
 \|e^{itL_a}g\|_{L^{10/3}([0,N^{-4/5-\eta}]\times\T^3)}
 \lesssim_{a,\eta}\|g\|_2
 \label{eq:cap-compact}
\end{equation}
whenever $\supp\widehat g\subseteq\{|k|\le3N\}$.

Set
\[
 A_*\coloneqq\max_{1\le j\le3}a_j,\qquad
 \theta_j\coloneqq\frac{a_j}{A_*},\qquad
 \alpha_j\coloneqq\theta_j^{-1/2}.
\]
After the fixed time change $t'=A_*t$, the phase becomes
\[
 t\sum_{j=1}^3a_jk_j^2
 =t'\sum_{j=1}^3\frac{k_j^2}{\alpha_j^2}.
\]
Under the coordinate change $z_j=\alpha_jx_j$, this is the standard
Schr\"odinger phase on the rectangular torus
\[
 \T^3_\alpha
 \coloneqq
 \prod_{j=1}^3\bigl(\R/(2\pi\alpha_j\Z)\bigr).
\]
The spatial change of variables modifies the $L^2$ and spacetime
$L^{10/3}$ norms only by constants depending on $a$.

We apply \cite[Theorem~1.1]{Quinn2026} on $\T^3_\alpha$.  In Quinn's notation
the spacetime dimension is $n=4$, and hence
\[
 q_c=\frac{2(n+1)}{n-1}=\frac{10}{3}.
\]
Let $G(z)\coloneqq g(z_1/\alpha_1,z_2/\alpha_2,z_3/\alpha_3)$ and
$P_\alpha\coloneqq\sqrt{-\Delta_{\T^3_\alpha}}$. Since $\alpha_j\ge1$,
$G$ has spectral support in $[0,3N]$. Take $\lambda=4N$ and a real even
$\psi\in C_c^\infty(\R)$ which equals one on $[-1,1]$ and is supported
in $[-2,2]$. Then $\psi(P_\alpha/\lambda)G=G$, as required by the cited
theorem. The sign of the propagator is immaterial by complex conjugation. Set
\[
 \varepsilon=\frac{\eta}{4},
 \qquad
 \delta=\lambda^{-1/5+\eta/2}.
\]
For all sufficiently large $N$,
\[
 \delta\ge\lambda^{-1/5+\varepsilon},
\]
and Quinn's theorem gives a lossless estimate on a time interval of length
\[
 (\lambda\delta)^{-1}
 =\lambda^{-4/5-\eta/2}.
\]
After undoing the fixed time rescaling, this interval contains
$[0,N^{-4/5-\eta}]$ for all sufficiently large $N$, since
$A_*N^{-4/5-\eta}/\lambda^{-4/5-\eta/2}\to0$.  Hence \eqref{eq:cap-compact} follows from
\cite[Theorem~1.1 and Section~9]{Quinn2026}. The finitely many remaining
frequencies are absorbed into the constant.

\Propref{prop:transfer} now yields the waveguide lossless upper bound
for $N^{-2}\le T\le N^{-4/5-\eta}$.  For $0<T\le N^{-2}$,
\eqref{eq:bernstein} gives the upper bound
$(N^2T)^{3/10}$, while \eqref{eq:coherent-lower} gives the matching lower
bound.  The same lower bound equals a positive constant when
$N^{-2}\le T\le N^{-4/5-\eta}$.  This proves
\eqref{eq:elliptic-short}.
\end{proof}

\medskip
\noindent{\bf AI disclosure.}
We used ChatGPT and Codex to discuss proof strategies, draft and check parts of the arguments, organize references, and improve the clarity and readability of the exposition. The authors take responsibility for the content of this manuscript.

\medskip
\noindent{\bf Funding }
Not applicable.
\medskip

\noindent{\bf Data Availability} This manuscript has no associated data. 

\medskip
\noindent{\bf\Large Declarations}
\medskip

\noindent\textbf{Conflicts of Interest}\space The authors declare no conflicts of interest.

\end{document}